\documentclass{article}
\usepackage{epsfig,pict2e}
\usepackage{fancyhdr}
\usepackage[tc]{titlepic}
\usepackage{amsmath,amsfonts,amssymb,relsize,geometry}
\usepackage{amsthm}
\usepackage{mathrsfs}
\usepackage[section]{placeins}
\usepackage[]{algorithm2e}
\usepackage{epsf}
\usepackage{epstopdf}
\usepackage{float}
\usepackage{subcaption}

\fancyheadoffset{1in}
\newtheorem{define}{Definition}[section]
\newtheorem{lem}[define]{Lemma}
\newtheorem{thm}[define]{Theorem}
\newtheorem{cor}[define]{Corollary}
\newtheorem{alg}[define]{Algorithm}

\begin{document}
%\setlength{\voffset}{-0.5in}
%\setlength{\headwidth}{18cm}
%\setlength{\headheight}{36pt}
%\fancyhead{\includegraphics[width=18cm]{banner.jpg}}
\title{\bf An algorithm for finding Hamiltonian Cycles in Cubic Planar Graphs}
\date{}
\author{\bf Bohao Yao \\ \bf Charl Ras, Hamid Mokhtar \\ \bf The University of Melbourne}
\maketitle
%\thispagestyle{fancy}
%\newpage
%\setlength{\headheight}{12pt}

\begin{abstract}
We first prove a one-to-one correspondence between finding Hamiltonian cycles in a cubic planar graphs and finding trees with specific properties in dual graphs. Using this information, we construct an exact algorithm for finding Hamiltonian cycles in cubic planar graphs. The worst case time complexity of our algorithm is O$(2^n)$.
\end{abstract}
	
\section{Introduction}
A {\it Hamiltonian cycle} is a cycle which passes through every vertex in a graph exactly once. A {\it planar graph} is a graph which can be drawn in the plane such that no edges intersect one another. A {\it cubic graph} is a graph in which all vertices have degree 3. Finding a Hamiltonian cycle in a cubic planar graph is proven to be an $\mathcal{NP}$-Complete problem \cite{garey1976planar}. This implies that unless $\mathcal{P}=\mathcal{NP}$, we could not find an efficient algorithm for this problem.
%It was further shown that the Hamiltonian cycle problem for grid graph is $\mathcal{NP}$-complete \cite{itai1982hamilton}.
Most approaches to finding a Hamiltonian cycle in planar graph utilises the {\it divide-and-conquer} method, or its derivation, the {\it separator theorem} which partitions the graph in polynomial time \cite{lipton1979separator}. Exact algorithms using such methods were found to have the complexity of O$(c^{\sqrt{n}})$ \cite{klinz2006exact} \cite{dorn2005efficient}, where {\it n} denotes the number of vertices and {\it c} is a constant. %Furthermore, a polynomial time algorithm was found for 4-connected graphs \cite{chiba1989hamiltonian}.
In this paper, we consider only cubic planar graphs and attempt to find a new algorithm to provide researchers with a new method to approaching this problem.

\section{The Expansion Algorithm}
We first start by introducing our so-called {\it Expansion Algorithm} which increases the number of vertices in a cycle at each iteration. A cycle can be first found by taking the outer facial cycle of the planar graph. We define it as the {\it base cycle}, $\sigma_0$. This base cycle is then expanded by the Expansion Algorithm, which will be described in detail later.

\begin{define}
\label{def1}
Consider a planar graph $G=(V,E)$

A \emph{complementary path}, $P_e^\sigma$, is a path between 2 adjacent vertices, $v_1, v_2 \in \sigma$ connected by the edge, $e$, s.t. $P_e^\sigma$ is internally disjoint from $\sigma$.

Furthermore, $P_e^\sigma$ and $e$ together will form the boundary of a face in $G$.
\end{define}

Assuming we are not dealing with multigraphs, the complementary path will always have at least one other vertex besides $v_1, v_2$.

The restriction that $P_e^\sigma$ and $e$ have to form the boundary of a face will be used later to prove Corollary \ref{cor1}.

\begin{define}
\label{def2}
Let $G_1=(V_1,E_1)$ and $G_2=(V_2,E_2)$. Then, $G_1+G_2 := (V_1 \cup V_2,E_1 \cup E_2)$
\end{define}

\begin{alg}[Expansion Algorithm]
\label{alg1}
\end{alg}

\begin{algorithm}
{\bf Let} $\sigma_0$ be the outer facial cycle. \\
{\bf Let} $i=0$\;
  \eIf{$\exists P_{e_i}^{\sigma_i}$}{
   $\sigma_{i+1} := \sigma_i + P_{e_i}^{\sigma_i} - e_i$\;
   $i=i+1$\;
   }{
   {\bf Output} $\sigma_i$\
  }
\end{algorithm}

At each iteration, the algorithm removes an edge, $e$ and adds a path $P_e^\sigma$. Since there are no vertices on $e$ and there is at least 1 vertex on $P_e^\sigma$, the number of vertices on the cycle will always increase at each iteration. Since there is only a finite amount of vertices in the graph, the algorithm will have to terminate eventually.

\begin{figure}[H]
\centering
\includegraphics{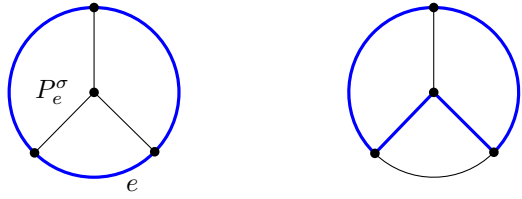}
\caption{Example of utilizing the Expansion Algorithm with the base cycle in blue}
\end{figure}

\begin{define}
The \emph{interior} of a cycle, $C$, is the connected region lying to the left of an anticlockwise orientation of $C$.
\end{define}

\begin{lem}
\label{lem2}
At each iteration of the Algorithm \ref{alg1}, all the vertices of G either lie in the interior of $\sigma$, or on $\sigma$.
\end{lem}

\begin{proof}
Let $P_i$ be the statement that all the vertices of G either lie in the interior of $\sigma_i$
\begin{itemize}
	\item $P_0$ is true as $\sigma_0$ is the outer facial cycle, with all the vertices either lie in the interior or on $\sigma_0$.
	\item Assume $P_k$ true. Then all the vertices either lie in the interior or on $\sigma_k$. Assume $\exists P_{e_k}^{\sigma_k}$, then $\sigma_{k+1}$ exists. Let $f_k$ be the face bounded by $P_{e_k}^{\sigma_k}$ and $e_k$. The interior of $f_k$ originally lies inside $\sigma_k$. But by an iteration of Algorithm \ref{alg1}, the interior of $f_k$ now lies outside the new base cycle, $\sigma_{k+1}$. However, since there are no vertices in the interior of a face, all the vertices in $\sigma_{k+1}$ also lies either in the interior or on $\sigma_{k+1}$. Therefore, if $\sigma_{k+1}$ exists, then $P_k \Rightarrow P_{k+1}$.
\end{itemize}
Hence, by mathematical induction, $P_i$ is true $\forall i$ as long as $\sigma_i$ exists.
\end{proof}
\begin{cor}
\label{cor1}
If $P_e^\sigma$ exists for an edge $e \in \sigma$, then $P_e^\sigma$ is unique.
\end{cor}

\begin{proof}
By Lemma \ref{lem2}, none of the vertices can lie outside $\sigma$, thus, $P_e^\sigma$ also lies in the interior of $\sigma$. Any edge $e$ lies on the boundary of 2 faces.  If $e \in \sigma$, then $\exists$ only 1 possible $P_e^\sigma$ such that $e$ and $P_e$ forms the boundary of a face that lies in the interior of $\sigma$ (the other face that $e$ is a boundary of lies outside $\sigma$)
\end{proof}

\begin{lem}
\label{lem1}
If $G$ have a Hamiltonian cycle, then $\exists$ a choice of complementary paths that algorithm \ref{alg1} can use to find that Hamiltonian cycle.
\end{lem}

\begin{proof}
Consider a Hamiltonian cycle, $C$, in $G$. If $C = \sigma_0$ then the case is trivial. If $C \ne \sigma_0$, then by Lemma \ref{lem2}, all vertices in $C$ lies either on or in the interior of $\sigma_0$ since $C$ contains all the vertices of the graph. Since $C$ contains all the vertices and $\sigma_0$ is a cycle and therefore must contain at least 3 vertices, $C \cap \sigma_0 \ne \emptyset$.

Suppose $C \ne \sigma_0$. Let $v_1,v_2,...,v_n$ be consecutive vertices on $\sigma_0$ about the clockwise rotation. Let $P_i$ be the path connecting $v_i,v_{i+1}$ that is the subpath of $C$. Let the edge between $v_i$ and $v_{i+1}$ be $e_i$ which also lies on $\sigma_0$. Since $C$ contains all the vertices in the graph, by iteratively finding $P_e^\sigma$, starting with $e = e_i$, more subpath of $P_i$ will lie on $\sigma$ at each iteration. The algorithm will only move on when $P_i \subset \sigma$.

Repeating this process $\forall i$, we will eventually end up with $\sigma = C$, in which the algorithm will terminate.
\end{proof}
\begin{figure}
\centering
\includegraphics{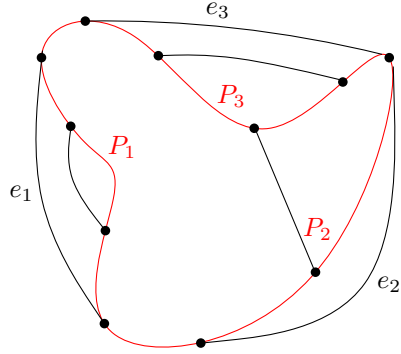}
\caption{Illustration of the proof for Lemma \ref{lem1} with $C$ in red}
\end{figure}

\section{The Problem in the Dual Graph}

\begin{define}
Given a cubic planar graph, $G$, and the corresponding dual graph, $\overline{G}=(\overline{V},\overline{E})$. A \emph{corresponding face}, $f_{\overline{v}} \in G$, is the face corresponding to the vertex $\overline{v} \in \overline{G}$.
\end{define}

\begin{define}
The \emph{outer vertex}, $\overline{v}^*$, is the vertex in $\overline{G}$ that corresponds to the outer face of $G$.
\end{define}

\begin{define}
Let $e \in G$ be the shared boundary between $f_{\overline{v}_1},f_{\overline{v}_2}$. A \emph{dual edge}, $\overline{e} \in \overline{G}$, is defined as an edge between $\overline{v}_1, \overline{v}_2 \in \overline{G}$.
\end{define}
\begin{figure}[H]
	\centering
	\includegraphics{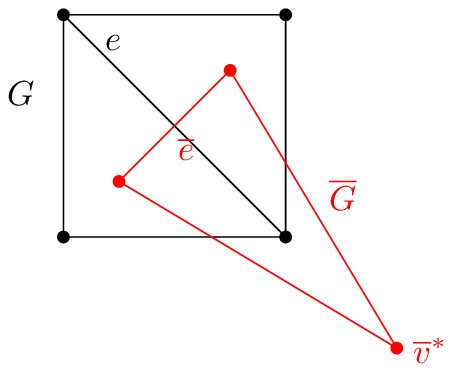}
	\caption{Example of $\overline{e}$ and $\overline{v}^*$}
\end{figure}

Consider Algorithm \ref{alg1}. Since $P_e^\sigma$ is unique for an edge $e$ by Corollary \ref{cor1}, let $e_0,e_1,...,e_n$ be the edges chosen in the {Expansion Algorithm} in order. Let $\overline{e}_0,\overline{e}_1,...,\overline{e}_n$ be the corresponding dual edges in $\overline{G}$.

\begin{figure}[H]
\centering
\includegraphics{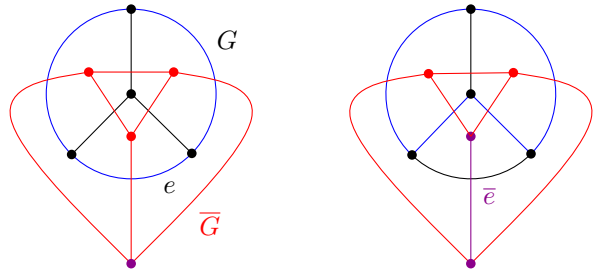}
\caption{Expansion Algorithm in the Dual Graph with $\overline{T}$ in Purple}
\end{figure}

\begin{lem}
\label{lem3}
$\bigcup\limits_{i=0}^n \overline{e}_i$ is a tree ($\overline{T}$).
\end{lem}

\begin{proof}
Let $P_a$ be the statement that $\bigcup\limits_{i=0}^a \overline{e}_i$ is a tree.
\begin{itemize}
\item $P_0$ is true as 2 vertices connected by an edge, $\overline{e}_0$ is considered a tree.
\item Assume $P_k$ is true. Then $\bigcup\limits_{i=0}^k \overline{e}_i$ is a tree. From proof of Corollary \ref{cor1}, $e_{k+1}$ is a boundary of 2 faces, one that lies inside $\sigma_{k+1}$ ($f_1$) and another that lies outside $\sigma_{k+1}$ ($f_2$). From proof of Lemma \ref{lem2}, by including path $P_{e_i}^{\sigma_i}$, the face bounded by $e_i$ and $P_{e_i}^{\sigma_i}$ which originally lies in the interior of $\sigma_i$ will now lie on the exterior of $\sigma_{i+1}$. Therefore, the vertex in $\overline{G}$ that corresponds to $f_1$ will lie outside $\overline{T}$ while the vertex that corresponds to $f_2$ will lie inside $\overline{T}$. Thus, $\overline{e}_{k+1}$ will connect a vertex on $\overline{T}$ to a vertex that is outside of $\overline{T}$, thus expanding the tree. Hence $\bigcup\limits_{i=0}^{k+1} \overline{e}_i$ is also a tree and therefore $P_{k+1}$ is true. Thus, $P_k \Rightarrow P_{k+1}$.
\end{itemize}

By mathematical induction, $\bigcup\limits_{i=0}^n \overline{e}_i$ is a tree.
\end{proof}

From Lemma \ref{lem2}, we know that faces $f_k$, bounded by $e_k$ and $P_{e_k}^{\sigma_k}$, will be on the exterior of $\sigma$ if $e_k$ was used in the algorithm. Conversely, if $e_k$ was not used in the algorithm, then $f_k$ will be in the interior of the cycle $\sigma$. This is equivalent to stating that all the vertices in $\overline{T}$ lies outside $\sigma$ while all the vertices not in $\overline{T}$ lies in the interior or $\sigma$.

\begin{lem}
\label{lem4}
If $\overline{v} \in \overline{G}$ lies on $\overline{T}$, then all the vertices in $G$ on $f_{\overline{v}}$ lies on $\sigma_n$.
\end{lem}

\begin{proof}
If $\overline{v}$ lies on the tree, then $\exists P_e^\sigma$ s.t. all its vertices in $G$ on $f_{\overline{v}}$ lies on $P_e^\sigma$ since the face is bounded by $P_e^\sigma$ and $e$. Since each iteration of the algorithm will only remove an edge and not any vertices, all the vertices already on the cycle will remain on the cycle at each iteration.
\end{proof}

%\begin{define}
%In a planar graph, 2 vertices are {\bf consecutive} about vertex $v$ if their ordering is consecutive when taking the clockwise orientation of all the vertices adjacent to $v$.
%\end{define}

\begin{thm}
\label{thm1}
$G$ has a Hamiltonian cycle if and only if $\exists \overline{T}$ found by Algorithm \ref{alg1}, with $\overline{v}^*$ at the root, that satisfy the following properties:

\begin{enumerate}
\item For each vertex $v \in G$ that lies on the boundaries of $f_{\overline{v}_1},f_{\overline{v}_2},f_{\overline{v}_3}$, at least one of the vertices, $\overline{v}_1,\overline{v}_2,\overline{v}_3 \in \overline{G}$, lies on $\overline{T}$.
\item No two vertices on $\overline{T}$ can be joined by an edge $\overline{e}$ unless $\overline{e} \in \overline{T}$.
\end{enumerate}
\end{thm}

\begin{proof}
By first taking the outer facial cycle, $\sigma_0$, only $\overline{v}^*$ lies outside our base cycle. Hence, $\overline{T}$ starts with $\overline{v}^*$ as its root.

($\Rightarrow$) We will begin by proving each of the following property:
\begin{enumerate}
\item If $G$ has a Hamiltonian cycle, then $\exists$ a cycle in $G$ that uses every vertex in $G$. If none of the vertices, $\overline{v}_1,\overline{v}_2,\overline{v}_3$, lies on $\overline{T}$, then none of the $P_e^\sigma$ found by Algorithm \ref{alg1} could contain $v$ from negation of Lemma \ref{lem4}. This contradicts the definition of a Hamiltonian cycle. Therefore, at least one of the vertices lies on $\overline{T}$.
\begin{figure}[H]
\centering
\includegraphics{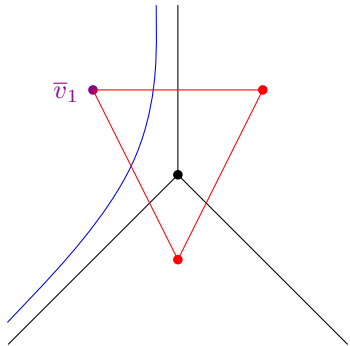}
\caption{A vertex lies on the cycle (blue) as $\overline{v}_1 \in \overline{T}$ (purple)}
\end{figure}
\item If 2 vertices, $\overline{v}'_1, \overline{v}'_2 \in \overline{T}$ are adjacent, $\exists$ 2 vertices shared by $f_{\overline{v}'_1},f_{\overline{v}'_2} \in G$. However,  the edge between them is not used on $\overline{T}$. This implies that the 2 vertices will lie on the corresponding path in $G$ twice, once from each face $f_{\overline{v}'_1},f_{\overline{v}'_2} \in G$. This is a contradiction from the definition of a Hamiltonian cycle.
\begin{figure}[H]
\centering
\begin{subfigure}{.5\textwidth}
  \centering
  \includegraphics[width=.4\linewidth]{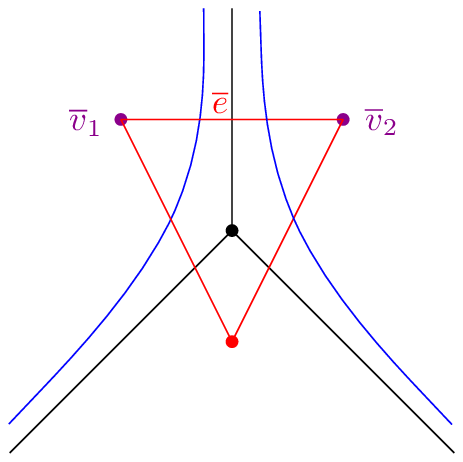}
  \caption{$\overline{v}_1,\overline{v}_2 \in \overline{T}$ but $\overline{e} \notin \overline{T}$}
\end{subfigure}
\begin{subfigure}{.5\textwidth}
  \centering
  \includegraphics[width=.4\linewidth]{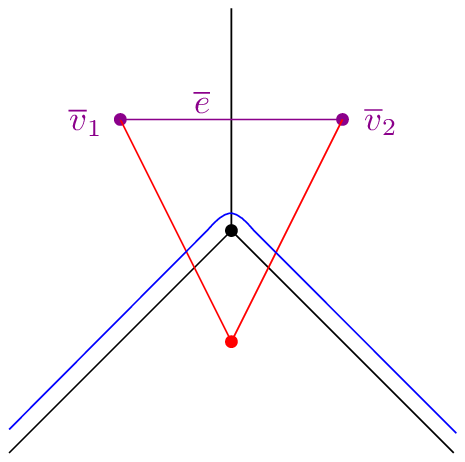}
  \caption{$\overline{v}_1,\overline{v}_2 \in \overline{T}$ and $\overline{e} \in \overline{T}$}
\end{subfigure}
\end{figure}
\end{enumerate}

($\Leftarrow$) Suppose that $\overline{T}=\bigcup\limits_{i=0}^n \overline{e}_i$ is a tree found by Algorithm \ref{alg1} that $\overline{T}$ satisfies those properties. From 2, we apply restriction on $\overline{T}$ such that the path in $G$ passes through each vertex at most once (as proven previously). From 1, we know the path passes through every vertex in $G$ at least once.
%By Algorithm \ref{alg1}, we know that we will always have a cycle regardless of the choice of $e$ as long as they exist. From the proof of Lemma \ref{lem3}, every $\overline{e}_{i+1} \in \overline{T}$ connects a vertex corresponding to a face outside $\sigma_i$ to a vertex corresponding to a face inside $\sigma_i$. Therefore, for every $\overline{e}$, $\exists$ a corresponding dual edge, $e$, that exists as a choice in Algorithm \ref{alg1}. Therefore, for all possible $\overline{T}$, we know that there exists a cycle in $G$.

Since there is a path in $G$ that passes through every vertex at most once and at least once, we know the path passes through every vertex in $G$ once. Furthermore, we know that Algorithm \ref{alg1} will always output a cycle, $\sigma$. Therefore, we know that $\overline{T}$ satisfying those properties will correspond to a Hamiltonian cycle in $\overline{G}$
\end{proof}

Note that Theorem \ref{thm1} is a corollary of a theorem by Skupie{\'n} \cite{skupien2002hamiltonicity}.

\section{Finding $\overline{T}$}
We start solving the problem by modifying a backtracking algorithm (\textbf{procedure} \texttt{solve}) with restrictions (\textbf{procedure} \texttt{update}) to limit the amount of guessing required. Restrictions from \textbf{procedure} \texttt{update} may cause the tree to form 2 disjoint graphs, in which case a new backtracking algorithm (\textbf{procedure} \texttt{disjoint}) is called to trace a path between the graphs. If a particular choice of paths does not work, the algorithm will negate that choice (\textbf{procedure} \texttt{backtrack}). The algorithm will terminate if it successfully find $\overline{T}$ (all verties assigned) or if all other choices are exhausted, concluding that the graph has no Hamiltonian cycle.

\subsection{Search Algorithm}
\begin{alg}
\label{alg2}
To find $\overline{T}$ with properties in Theorem \ref{thm1}, we use a modified backtracking algorithm.
\end{alg}

\begin{algorithm}[H]
{\bf procedure} update($\overline{G},\overline{v}$): \\
\If{$\exists$ an unassigned vertex, $\overline{v}_1$, adjacent to both $\overline{v}$ and any other vertex, $\overline{u} \in \overline{T}$, s.t. $\exists$ a path between $\overline{u}$ and $\overline{v}$ \tcp*{See comments (*)}}{
$\overline{S}:=\overline{S}+\overline{v}_1$}
\tcc{Satisfies Property 2 of Theorem \ref{thm1}}
\If{$\exists$ 2 \emph{adjacent} vertices, $\overline{v}_1, \overline{v}_2 \in \overline{S}$\;}{
$V(\overline{T}):=V(\overline{T})+\overline{v}_3$, where $\overline{v}_3$ is all the vertices adjacent to both $\overline{v}_1$ and $\overline{v}_2$\;}
\tcc{Satisfies Property 1 of Theorem \ref{thm1}}
\eIf{$\exists$ a vertex $\in \overline{S} \cap \overline{T}$}{
return false \tcp*{A vertex cannot lie in both $\overline{S}$ and $\overline{T}$}}{
return true}
\end{algorithm}

\begin{algorithm}[H]
{\bf procedure} backtrack($\overline{G},\overline{v}$): \\
$V(\overline{T}):= V(\overline{T}) - \overline{v}$\;
$\overline{S}:=\overline{S}+\overline{v}$\;
\tcc{If $\exists$ a contradiction, negate the assumption made about $\overline{v}$}
\end{algorithm}

\begin{algorithm}[H]
{\bf procedure} solve($\overline{G}$): \\
\eIf{all vertices $\in \overline{G}$ assigned}{
return true}{
\For{each unassigned vertex, $\overline{v}$, adjacent to any vertex $\overline{w} \in \overline{T}$}{
$\overline{T}:=\overline{T}+(\overline{v},\overline{w})$ \tcp*{guess $\overline{v} \in \overline{T}$}
\eIf{update$(\overline{G},\overline{v})=$ false}{
return false \tcp*{Choice leads to contradiction}}{
\eIf{$\exists$ a disconnected vertex, $\overline{u} \in \overline{T}$}{
\If{disjoint$(\overline{G},\overline{u})$ return false}{
call backtrack($\overline{G},\overline{v}$)}}
{\eIf{solve($\overline{G}$) return true}{
output $\overline{T}$ \tcp*{Succeeded in finding $\overline{T}$}}{
call backtrack($\overline{G},\overline{v}$)}}}}}
\end{algorithm}
\newpage
\begin{algorithm}
\FloatBarrier
{\bf procedure} disjoint($\overline{G},\overline{v}$): \\
\eIf{all vertices $\in \overline{G}$ assigned}{
return true}{
\eIf{$\nexists$ a path between $\overline{v}$ and $\overline{v}^*$}{
\For{each unassigned vertex, $\overline{u}$, adjacent to $\overline{w} \in \overline{T}$, s.t. $\exists$ a path between $\overline{w}$ and $\overline{v}$}{
$\overline{T}:=\overline{T}+(\overline{u},\overline{w})$ \tcp*{guess $\overline{u} \in \overline{T}$}
\eIf{update$(\overline{G},\overline{u})=$ false}{
return false}{
\If{disjoint$(\overline{G},\overline{u})$ returns false\;}{
call backtrack($\overline{G},\overline{u}$)}}}}{
\eIf{solve($\overline{G},\overline{v}$) returns true\;}{
output $\overline{T}$}{
call backtrack($\overline{G},\overline{v}$)}}}
\end{algorithm}

\begin{algorithm}
\FloatBarrier
{\bf Begin} \\
$\overline{S} := \emptyset$\;
$\overline{T} := \left\{\overline{v}^*\right\}$\;
\tcp{Define set $\overline{S}$ as vertices that cannot lie on $\overline{T}$}
call solve($\overline{G}$)
\end{algorithm}

\noindent * Restrictions on vertices does not apply in disjoint graphs where the 2 vertices $\in \overline{T}$ belongs to different sets due to the necessity of connecting the disjoint graphs into one tree.
\newpage
\subsection{Complexity}
In the worst case scenario, we expect to force at least one vertex into $\overline{S}$ at each guess. This reduces the number of faces in which we have to check to $\frac{f}{2}$ where $f$ is the number of faces in $G$. Since we are working in planar graph, $e=\frac{3}{2}n$ where $n$ is the number of vertices. Using Euler's formula:

\begin{align*}
n-e+f&=2 \\
n-\frac{3}{2}n+f&=2 \\
f&=2+\frac{n}{2}
\end{align*}

Since we have 2 choices ($\overline{T}$ or $\overline{S}$) that we have to guess, in the worst case, our algorithm runs in $2^{\frac{f}{2}}$ or $2^{1+\frac{n}{4}}$. Hence, our algorithm has a time complexity of O$(2^n)$.

\section{Future Research}
%\subsection{Inclusion of degree 2 vertices}
%Let $G_1=(V_1,E_1)$ be a planar graph with all its vertices having degree 3 or lower. Let $\overline{G}_1=(\overline{V}_1,\overline{E}_1)$ %be the corresponding dual graph. Notice that when converting from $G$ to $\overline{G}$, all the vertices with degree 2 will be ignored.

%\begin{lem}
%If the corresponding edge, $\overline{e} \in \overline{G}_1$, cannot lie on $\overline{T}$ if $\exists$ at least 1 vertices of degree 2 in %its primal, $e \in G_1$ (note that $e$ is an edge between 2 vertices of degree 3 since vertices with degree 2 are ignored when converted to %dual).
%\end{lem}

%\begin{proof}
%By Algorithm \ref{alg1}, we find a path, $P_e^\sigma$, between 2 adjacent vertices connected by the edge, $e$. However, if $\exists$ a %vertex of degree 2 $\in e$, that path would not lie between adjacent vertices. Hence, $\overline{e}$ cannot be used in Algorithm \ref{alg1} %and thus cannot lie in $\overline{T}$.
%\end{proof}
%\subsection{Non-cubic planar graphs}
We could modify Theorem \ref{thm1} such that it encompasses any planar graph, $G_2=(V_2,E_2)$ However, caution is advised dealing with faces that share a single vertex, but not a boundary, as this is not reflected in the dual graph, $\overline{G}_2=(\overline{V}_2,\overline{E}_2)$. To overcome this, we introduce the \emph{imaginary edge}.

\begin{define}
An \emph{imaginary edge}, $\overline{e}'$ between vertices in $\overline{G}_2$ if the corresponding faces in $G_2$ share a single vertex but not a boundary.
\end{define}

\begin{thm}
\label{thm2}
If $G_2$ has a Hamiltonian cycle then $\exists \overline{T}$ found by Algorithm \ref{alg1}, with $\overline{v}^*$ at the root, that at least satisfy the following properties:

\begin{enumerate}
\item For each n degree vertex $v \in G_2$ that lies on the boundaries of $f_{\overline{v}_1},f_{\overline{v}_2}, ..., f_{\overline{v}_n}$, at least one of the vertices, $\overline{v}_1,\overline{v}_2, ..., \overline{v}_n \in \overline{G}_2$, lies on $\overline{T}$
\item No two vertices on $\overline{T}$ can be joined by an edge $\overline{e}$ unless $\overline{e} \in \overline{T}$.
\item No two vertices joined by an imaginary edge can lie on $\overline{T}$
\end{enumerate}
\end{thm}

\begin{proof}
Properties 1 and 2 are proven in Theorem \ref{thm1}.

The proof of property 3 follows from the proof of Property 2:

If 2 vertices, $\overline{v}_1, \overline{v}_2 \in \overline{T}$ are joined by an imaginary edge, then the vertex in which $f_{\overline{v}_1}$ and $f_{\overline{v}_2}$ shares will be on the cycle twice. Since this contradicts the definition of a Hamiltonian cycle, property 3 holds.
\end{proof}

\section{Acknowledgments}
The author wishes to thank his supervisors, Dr Charl Ras and Hamid Mokhtar, for their continued guidance and providing useful insights to this problem.

This paper was supported by the Vacation Research Scholarship awarded by the Australian Mathematical Sciences Institute.


\begin{thebibliography}{1}

  \bibitem{dorn2005efficient} F. Dorn, E. Penninkx, H. L. Bodlaender, and F. V. Foin. Efficient exact algorithms on planar graphs: Exploiting sphere cut branch decompositions. In {\em Algorithms-ESA 2005}, pages 95-106. Springer, 2005.

  \bibitem{garey1976planar} M. R. Garey, D. S. Johnson, and R. E. Tarjan. The planar hamiltonian circuit problem is NP-complete. {\em SIAM Journal of Computing}, 5(4):704-714, 1976.

  \bibitem{klinz2006exact} B. Klinz, G. J. Woeginger, et al. Exact algorithms for the Hamiltonian cycle problem in planar graphs. {\em Operations Research Letters}, 34(3):269-274, 2006.

  \bibitem{lipton1979separator} R. J. Lipton and R. E. Tarjan. A separator theorem for planar graphs. {\em SIAM Journal on Applied Mathematics}, 36(2):177-189, 1979.
  
    \bibitem{skupien2002hamiltonicity} Z. Skupie{\'n}. Hamiltonicity of planar cubic multigraphs. {\em Discrete mathematics}, 251(1):163-168, 2002.

  \end{thebibliography}
\end{document}